\documentclass[11pt]{amsart}
\usepackage{latexsym, graphicx, amssymb, cite, hyperref, tikz-cd,amsmath}
\usepackage[normalem]{ulem}
\usetikzlibrary{angles, quotes} 
\usetikzlibrary{arrows}
\usepackage{subcaption}
\numberwithin{equation}{section}

\newtheorem{theorem}{Theorem}[section]

\newtheorem{lemma}[theorem]{Lemma}

\newtheorem{corollary}[theorem]{Corollary}

\newtheorem{proposition}[theorem]{Proposition}

\theoremstyle{definition}

\newtheorem{remark}[theorem]{Remark}

\def\be{\begin{eqnarray}}
\def\ee{\end{eqnarray}}
\def\beal{\begin{aligned}}
\def\enal{\end{aligned}}

\usepackage[utf8]{inputenc}
\usepackage[english]{babel}

\usepackage[T1, T2A]{fontenc}

\def\R{\mathbb R}
\def\N{\mathbb N}

\newtheorem*{Pf}{Proof}

\newcommand\beqa[1]{ \begin{eqnarray} \label{#1}}
\newcommand{\eeqa}{ \end{eqnarray} }
\newcommand{\beqano}{ \begin{eqnarray*} }
\newcommand{\eeqano}{ \end{eqnarray*} }

\def\beal{\begin{aligned}}
\def\enal{\end{aligned}}

\newcommand \x {\xi}

\renewcommand \phi {\varphi}

\newcommand \rank {{\rm rank\ }}

\def\~ {\tilde}
\def\Bbb{\mathbb}

\def\R{\Bbb R}

\newcommand{\inj}{\operatorname{inj}}

\newcommand \codim {\operatorname{codim}}

\title{Kuznecov remainders and generic metrics}

\author{Vadim Kaloshin}
\address{Institute of Science and Technology Austria, Klosterneuburg, Lower Austria} 
\email{Vadim.Kaloshin@gmail.com} 

\author{Emmett L. Wyman}
\address{Department of Mathematics and Statistics, Binghamton University, Vestal NY}
\email{ewyman@binghamton.edu}

\author{Yakun Xi}
\address{School of Mathematical Sciences, Zhejiang University, Hangzhou 310027, PR China}
\email{yakunxi@zju.edu.cn}

\begin{document}

\begin{abstract}
We obtain improved remainder estimates in the Kuznecov sum formula for period integrals of Laplace eigenfunctions on a Riemannian manifold $M$. Building upon the two-term asymptotic expansion established in \cite{2term}, we prove that for a Baire-generic class of metrics, the oscillatory second term in the Kuznecov formula can be eliminated, yielding an improved remainder estimate. 

\end{abstract}

\maketitle

\section{Introduction}

\subsection{Background}

Let $(M,g)$ be a compact Riemannian manifold without boundary. The Laplace-Beltrami operator is given in local coordinates by
\[
    \Delta_g = |g|^{-1/2} \partial_i (g^{ij} |g|^{1/2} \partial_j \ \cdot \ ).
\]
$L^2(M)$ admits an orthonormal basis of Laplace-Beltrami eigenfunctions $e_j$, $j \in \N$ with
\[
    \Delta_g e_j = -\lambda_j^2 e_j.
\]
We are interested in the asymptotic nature of eigenfunctions as an eigenvalue $\lambda_j \to \infty$. The key principle at work is: \emph{Spectral asymptotics are determined by the structure and dynamical properties of the geodesic flow.} There are many examples of this principle, one fundamental example being the result of Duistermaat and Guillemin \cite{DG} (see also \cite{Ivrii}) in which they obtain a qualitative improvement to the remainder term in the Weyl law under the assumption that the closed orbits of the geodesic flow comprise a measure-zero subset of the cotangent bundle.

Other interesting quantities include the so-called \emph{period integrals}
\[
    \int_\Sigma e_j \, dV_\Sigma
\]
of eigenfunctions over some closed embedded submanifold $\Sigma \subset M$. In \cite{z92}, Zelditch obtains a Kuznecov sum formula of the form
\[
    N_\Sigma(\lambda) := \sum_{\lambda_j \leq \lambda} \left| \int_\Sigma e_j \, dV_\Sigma  \right|^2 = C_{\Sigma,M} \lambda^{\codim \Sigma} + O(\lambda^{\codim \Sigma - 1}),
\]
from which follow the bounds on the individual terms
\[
    \left| \int_\Sigma e_j \, dV_\Sigma \right| = O(\lambda^{\frac{\codim \Sigma - 1}{2}}).
\]
There was a great deal of activity around finding both little-$o$ and logarithmic improvements to the individual term bounds under various dynamical and geometric assumptions \cite{CSPer,Gauss,emmett1,emmett2,emmett3,emmett4,CGT,CanGal1,CanGal2,CanGal3,CanGal4}, but it was not until \cite{2term} that the remainder of Zelditch's asymptotic formula was improved. One of the difficulties is that the improvement of the remainder revealed a structured oscillating term of order $\lambda^{\codim \Sigma - 1}$, which needed to be characterized.

Let 
\[
  G^t: S^*M \longrightarrow S^*M
\]
denote the time‑$t$ homogeneous geodesic flow on the unit cotangent bundle $S^*M$.
For a closed, embedded submanifold $\Sigma\subset M$, we denote by $SN^*\Sigma$ the unit conormal subbundle of $\Sigma$.
 Here is an abridged version of the results of \cite{2term}. To start with, there are a few technical conditions on the geodesic flow, including the set of return times. We assume that
\begin{equation}
    \label{eq: def looping times}
    \mathcal T_0 = \{ t \neq 0 : G^t(SN^*\Sigma) \cap SN^*\Sigma \neq \emptyset \}
\end{equation}
is countable, and, say, the set of recurrent elements of $SN^*\Sigma$ is measure zero in $SN^*\Sigma$.
Then,
\begin{equation}
    \label{eq: 2term asymptotics}
    N_\Sigma(\lambda) = C_{\Sigma,M} \lambda^{\codim \Sigma} + Q(\lambda) \lambda^{\codim \Sigma - 1} + c_{\Sigma,M} + o(\lambda^{\codim \Sigma - 1})
\end{equation}
where $Q(\lambda)$ is a bounded, uniformly continuous function where
\[
    Q(\lambda) = \sum_{t \in \mathcal T_0} e^{-it\lambda}\  \frac{q(t)}{-it},
\]
where $q(t)$ is some integral over the subset
\begin{equation}
    \label{eq: looping direction set}
    \{ (x,\xi) \in SN^*\Sigma : G^t(x,\xi) \in SN^*\Sigma \} \subset SN^*\Sigma.
\end{equation}
One can use the precise description of $Q(\lambda)$ and $q(t)$ to show that $Q(\lambda)$ is nonzero for all distance spheres of sufficiently small radius.

As noted in \cite{2term}, the constant term $c_{\Sigma,M}$ is there to address a quirk in the proof for the codimension-$1$ case. This constant has been $0$ in all concrete examples that the authors have tried, so it might be possible to remove it in general. To be safe, we include it.

Theorem 1.4 of \cite{2term} is important for us here, so we state it for convenience.
\begin{theorem}\label{thm: 2term}
    If $\mathcal T_0$ is countable, and the set
    \begin{equation}\label{eq: measure-zero}
        \{(x,\xi) \in SN^*\Sigma : G^t(x,\xi) \in SN^*\Sigma \text{ for some } t > 0 \}
    \end{equation}
    is measure-zero in $SN^*\Sigma$, then
    \[
        N_\Sigma(\lambda) = C_{\Sigma,M} \lambda^{\codim \Sigma} + c_{\Sigma,M} + o(\lambda^{\codim \Sigma - 1}).
    \]
\end{theorem}

We will refer to the set \eqref{eq: measure-zero} as the set of \emph{looping directions} conormal to $\Sigma$.

\subsection{Statement of results}

The goal of this article is to show that the assumption of Theorem \ref{thm: 2term} is generic.
That is, given a smooth, compact manifold $M$ without boundary, and a closed embedded submanifold $\Sigma$, there is a topologically generic set of Riemannian metrics for which the hypotheses of Theorem \ref{thm: 2term} are satisfied. 

Let $\mathcal G$ denote the set of smooth Riemannian metrics on $M$ with the $C^\infty$ topology. This is a Fr\'echet space, and hence also completely metrizable. \emph{Residual} or \emph{comeager} sets in $\mathcal G$ will be considered generic. Recall that a set is said to be residual in $\mathcal G$ if it contains a countable intersection of open dense sets.

Given a Riemannian metric $g$ on $M$, we consider the symbol
\begin{equation}\label{eq: p_g}
    p_g(x,\xi) = \sum_{i,j} g^{ij}(x) \xi_i \xi_j.
\end{equation}
Note, $p_g$ is the principal symbol of the positive-definite Laplace-Beltrami operator $-\Delta_g$. The Hamilton vector field on the cotangent bundle $T^*M$ associated with the symbol $p_g$ is given by
\[
    H_{p_g} = \sum_j \left( \frac{\partial p_g}{\partial \xi_j} \frac{\partial}{\partial x_j} - \frac{\partial p_g}{\partial x_j} \frac{\partial}{\partial \xi_j} \right).
\]
The diffeomorphism $\exp(\frac 1 2 H_{p_g})$ on $T^*M$ is called the \emph{geometer's geodesic flow}. Our main result is as follows.

\begin{theorem}\label{thm: main}
    Let $M$ be a smooth, compact manifold without boundary, and let $\Sigma \subset M$ be a closed, embedded submanifold. Then there is a residual set of metrics $g$ in $\mathcal G$ for which
    \[
    \exp\bigl(\tfrac12 H_{p_g}\bigr)\colon \dot N^*\Sigma \to T^*M
    \]
    is transversal to $\dot N^*\Sigma$, where 
    $
    \dot N^*\Sigma = N^*\Sigma\setminus\{0\}
    $
    denotes the conormal bundle with the zero section removed.
\end{theorem}

    We use the conormal bundle $N^*\Sigma$ instead of the unit conormal bundle $SN^*\Sigma$ since the latter changes with the metric $g$ and the former does not.

As a corollary, we have:

\begin{corollary}\label{cor: main}
    Assume the hypotheses of Theorem \ref{thm: main}. Then, there is a residual set of metrics $g$ in $\mathcal G$ for which
    \[
        N_\Sigma(\lambda) = C_{\Sigma,M} \lambda^{\codim \Sigma} + c_{\Sigma,M} + o(\lambda^{\codim \Sigma - 1}).
    \]
\end{corollary}

The proof is short, so we include it here.

\begin{proof}[Proof of Corollary \ref{cor: main}]
Consider a metric $g$ from the output of Theorem \ref{thm: main}. We will show that, for this metric, the set of looping directions \eqref{eq: measure-zero} is countable. The corollary will then follow from Theorem \ref{thm: 2term}.

We start by rewriting the set \eqref{eq: measure-zero} of looping directions as the image of the projection
\[
    \{ (x,\xi) \in \dot N^*\Sigma : \exp\bigl(\tfrac 1 2 H_{p_g}\bigr)(x,\xi) \in \dot N^*\Sigma \} \to SN^*\Sigma,
\]
and so it suffices to show the set on the left is countable. However, the set on the left is the intersection
\[
    \exp\bigl(\tfrac 1 2 H_{p_g}\bigr)(\dot N^*\Sigma) \cap \dot N^*\Sigma.
\]
For our metric, this intersection is transversal. Since $\dot N^*\Sigma$ and its image under the flow are both of half the dimension of the ambient space $T^*M$, the points of the intersection are isolated, which implies the claim.
\end{proof}

We make a few remarks about the main result. First, the submanifold $\Sigma$ must be fixed beforehand. That is, it is not possible to obtain the improvement in Corollary \ref{cor: main} for \emph{all} $\Sigma$ generically, as this formula will have a nontrivial $Q$ term for small distance spheres, as previously remarked. Second, while the main result of this paper, Corollary \ref{cor: main}, says something about spectral asymptotics, all arguments in this paper are geometric (or perhaps dynamical) in nature.

For other genericity results having to do with spectral analysis, see e.g. \cite{cggeneric} and \cite{echolocation}. In the former, Canzani and Galkowski obtain logarithmic improvements to the remainder of the Weyl law for a generic set of metrics. Their notion of a generic set is \emph{predominant} set, which is larger than a residual set. In \cite{echolocation}, the second and third authors show that one can hear at which point a generic drum is struck---amongst drums shaped like compact manifolds without boundary.

\begin{figure}[htbp]
  \centering
  \begin{subfigure}[b]{\textwidth}
    \centering
    \includegraphics[height=0.35\textheight,keepaspectratio]{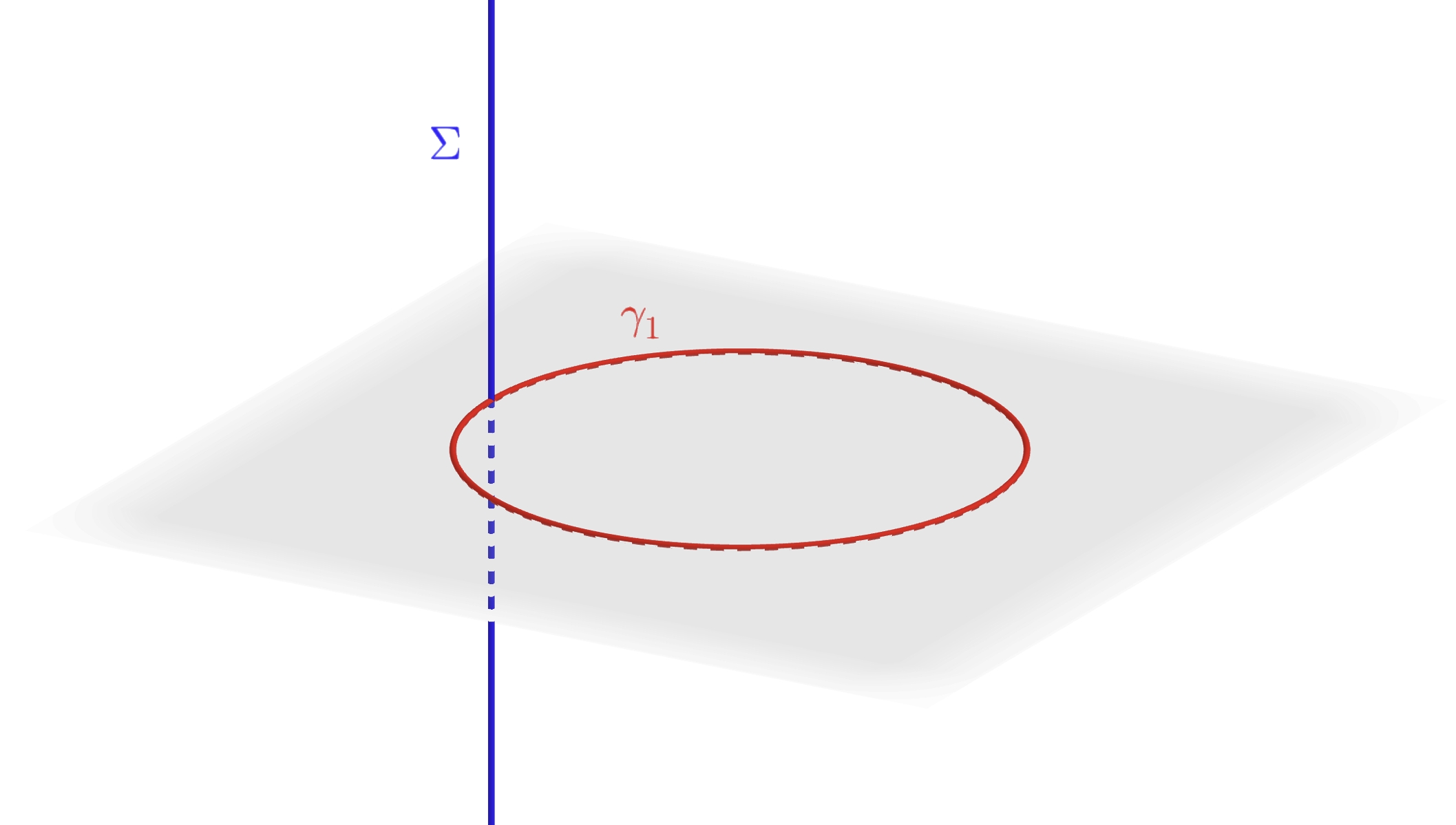}
    \caption{A closed geodesic normal to $\Sigma$.}
    \label{fig:gamma1}
  \end{subfigure}
  \vspace{1em}
  \begin{subfigure}[b]{\textwidth}
    \centering
    \includegraphics[height=0.38\textheight,keepaspectratio]{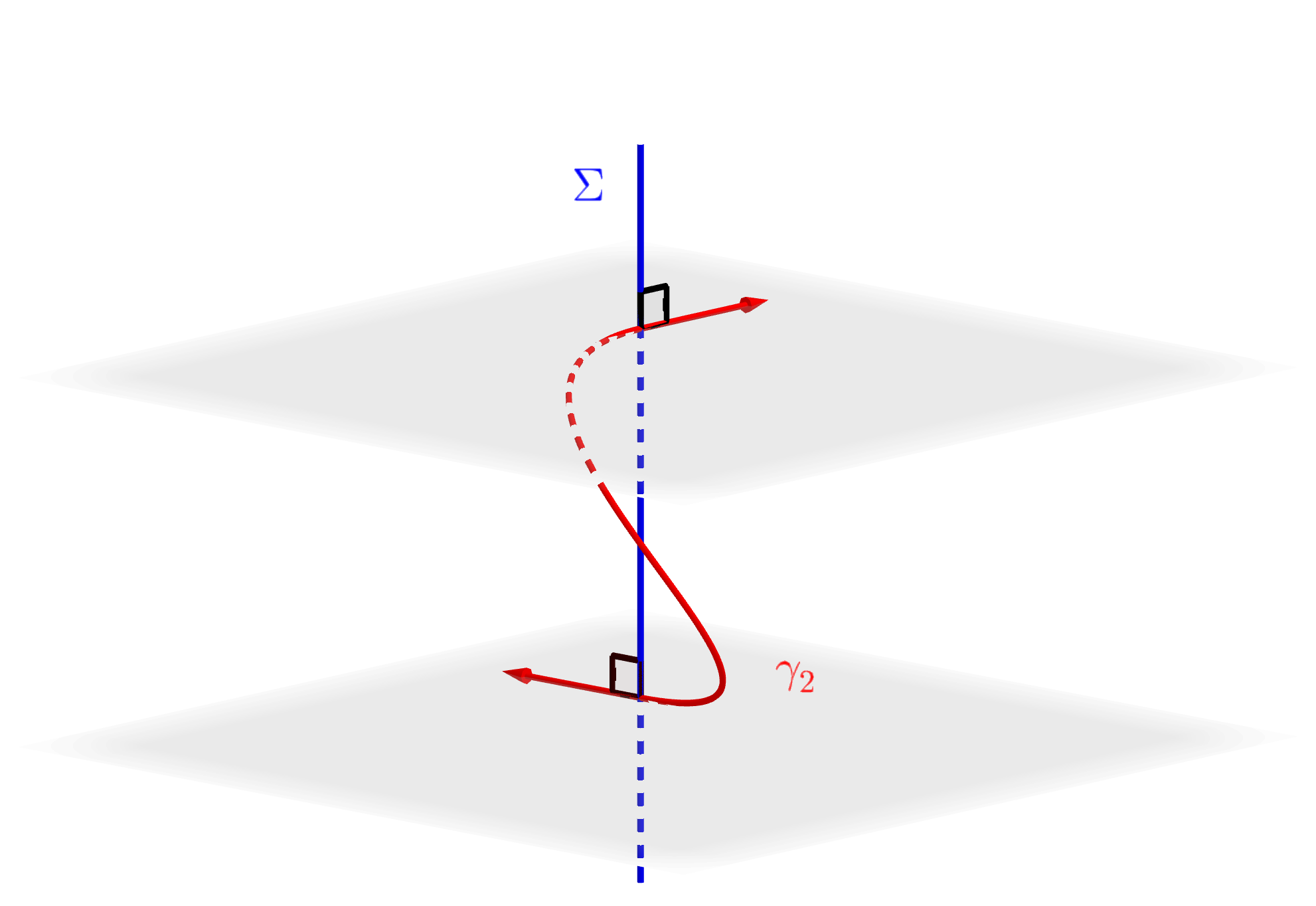}
    \caption{A looping geodesic of $\Sigma$ that is not closed.}
    \label{fig:gamma2}
  \end{subfigure}
  \caption{Two enemy scenarios for improving the Kuznecov remainder: in the top panel, $\gamma_1$ is a closed geodesic normal to $\Sigma$; in the bottom panel, $\gamma_2$ is a looping geodesic of $\Sigma$ that does not close.} 
  \label{fig: loop and closed}
\end{figure}

{As discussed in \cite{2term} and earlier works, two types of geodesic orbits play key roles in the analysis of Kuznecov sums. 
\begin{itemize}
\item \underline{Periodic geodesics} (top panel of Figure \ref{fig: loop and closed}) that meet the submanifold orthogonally give rise to recurrent conormal directions, contributing to individual terms in the sum;
    \item \underline{looping geodesics} (bottom panel of Figure \ref{fig: loop and closed}) return conormally to the submanifold at a later time, contributing to the Kuznecov remainder.
\end{itemize}

Since these phenomena have different geometric origins and encode distinct spectral information, it is natural to treat them separately. In particular, periodic geodesics normal to $\Sigma$ are more delicate to perturb—since they pass repeatedly through the same local neighborhood—whereas non-periodic looping geodesics can be handled more directly. Once both sources of large remainder terms are addressed independently, we obtain the improved Kuznecov remainder.

{Recall that a metric $g$ is \emph{bumpy} if all its closed geodesics are nondegenerate. By a classical theorem (see, e.g.,  \cite{anosov1983generic}), bumpy metrics are generic. More recently, Figalli and Rifford \cite{figalli2015closing1,figalli2015closing2} studied bumpy metrics in depth, obtaining them solely via conformal perturbations.
Observe that, under a bumpy metric there are only finitely many closed geodesics of uniformly bounded length. By perturbing the embedding of $\Sigma$ we can guarantee that none of those finitely many closed loops ever intersect it orthogonally.  (see Proposition \ref{prop: perturbation} and Figure \ref{fig: closed geodesic}). Once all periodic normals have been removed, the only remaining obstructions are non-periodic looping geodesics, which we eliminate via a second family of localized conformal perturbations that force the desired transversality.}

{We stress, however, that a bumpy metric by itself is insufficient to improve the Kuznecov remainder.
Indeed, as shown in \cite{2term}, if $\Sigma$ is a geodesic sphere then no metric (bumpy or not) admits an improvement. The transversality conditions and overall framework required by our proof are therefore of a different character than in the classical bumpy metric setting, necessitating the new techniques developed here.}

To implement our two‐step plan, we proceed as follows. We begin with a bumpy metric on $M$, which guarantees that every closed geodesic is nondegenerate. First, we construct diffeomorphism‐induced perturbations that modify the embedding of $\Sigma$, thereby ensuring that no periodic geodesic can meet $\Sigma$ orthogonally (see Figure \ref{fig: closed geodesic} in Section \ref{sec: no closed normal}). Once all periodic normals have been removed, only looping geodesics remain. To eliminate those, we introduce a class of localized conformal perturbations whose effect on the principal symbol forces the geodesic flow to be transverse to $\dot N^*\Sigma$. This transversality argument then yields the improved remainder estimate in the Kuznecov formula.

\subsection{Structure of the paper}

Section \ref{sec: bumpy metric} recalls several immediate consequences of the bumpy metric theorem and records some technical facts.  Section \ref{sec: perturbation} presents two distinct families of local metric perturbations, proves that each satisfies the required full-rank differential property, and constitutes the most technical part of this paper. Furthermore, Remark \ref{rmk: footballs} exhibits an example demonstrating that the conformal perturbation family alone is insufficient for our argument.
 In Section \ref{sec: no closed normal} we apply the first family to show that, for a residual set of metrics, there are no closed geodesics that intersect the fixed submanifold orthogonally.
Section \ref{sec: proof of main} employs the second family to establish transversality and thereby complete the proof of Theorem \ref{thm: main}.  Finally, Section \ref{sec: further problems} outlines several related problems and open questions suggested by our results.

\subsection*{Acknowledgements} Y. X. is supported by the National Key Research and Development Program of China No. 2022YFA1007200, Zhejiang Provincial Natural Science Foundation of China under
Grant No. LR23A010002, and NSF China Grant No.
12171424. V. K. is partially supported by ERC Grant \#885707. E. L. W. is partially supported by NSF grant DMS-2204397.

\section{Some quick consequences of the bumpy metric theorem}\label{sec: bumpy metric}

We need to make use of the following observation about the bumpy metric theorem, which is known to experts. In what follows, $\mathcal G$ will denote the Fr\'echet manifold of $C^\infty$ Riemannian metrics on our compact, boundary-less smooth manifold $M$.

\begin{proposition}\label{prop: bumpy metric}
    For each $T > 0$, there exists an open-dense subset of $\mathcal G$ for which there are only finitely many geodesics which close by time $T$.
\end{proposition}

This follows from the classical result of \cite{abraham1970bumpy,anosov1983generic} with some minor adaptations. We will first show how this result holds for $C^k$ metrics and then use an argument in \cite{white2017bumpy} to transfer it to $C^\infty$ metrics.

Let $ \mathcal G^k$ denote the space of $C^k$ Riemannian metrics on $M$ with $k \in \N$ or $k = \infty$. For $r > 0$, let
\[
     \mathcal G_r^k = \{g \in  \mathcal G^k : \inj(M,g) > r \}
\]
denote the subset of $C^k$ metrics with injectivity radius larger than $r$. We observe that:
\begin{enumerate}
    \item $ \mathcal G_r^k$ is an open subset of $ \mathcal G^k$.
    \item Since $M$ is compact, $ \mathcal G^k = \bigcup_{r > 0}  \mathcal G_r^k$.
\end{enumerate}
It follows almost immediately that:

\begin{lemma}\label{lem: injectivity radius}
    Let $k \in \N$ or $k = \infty$. Suppose $U \subset  \mathcal G^k$ and that $U \cap  \mathcal G_r^k$ is open and dense in $ \mathcal G_r^k$ for each $r > 0$. Then $U$ is open and dense in $ \mathcal G^k$.
\end{lemma}

\begin{proof}
    We see $U$ is open in $ \mathcal G^k$ by writing
    \[
        U = \bigcup_{r > 0} U \cap  \mathcal G_r^k.
    \]
    We see $U$ is dense by considering a nonempty open subset $V \subset  \mathcal G^k$ and noting $V \cap  \mathcal G^k_r$ is nonempty for some $r > 0$. Since $U \cap  \mathcal G^k_r$ is dense in $ \mathcal G^k_r$, $U \cap V \cap  \mathcal G_r^k$ is nonempty, and hence so is $U \cap V$.
\end{proof}

{
Recall, a \emph{Jacobi field} $J$ along a geodesic is solves
\[
    \frac{D^2}{dt^2} J + R(J,\gamma')\gamma' = 0,
\]
where $R$ is the Riemann curvature tensor and $\frac{D^2}{dt^2}$ is the second covariant derivative with respect to the parameter $t$ of the geodesic. Jacobi fields are the vector fields generated by a variation of a geodesic.}

\begin{lemma}\label{lem: open-dense bumpy metric}
    The set of metrics in $ \mathcal G^k$ for which there exist no closed geodesics of period in $(0,T]$ which admit nontrivial smooth Jacobi fields is open and dense.
\end{lemma}

\begin{proof}
    Let $U \subset  \mathcal G^k$ denote the set of metrics for which no closed geodesics of period in $(0,T]$ admit nontrivial Jacobi fields. By the classical bumpy metric theorem \cite{abraham1970bumpy,anosov1983generic}, $U$ is residual and hence dense in $ \mathcal G^k$. By the lemma, it suffices to show $U \cap  \mathcal G_r^k$ is open in $ \mathcal G_r^k$ for each $r > 0$.

    Let $E_r \subset  \mathcal G^k_r \times (0,T] \times S^*M$ denote the problematic set
    \[
        E_r = \{(g,t,x,\xi) : \Phi^t_g(x,\xi) = (x,\xi) \text{ and } \rank (I - d\Phi_g^t) > 1 \},
    \]
    where here $S^*M$ is the quotient of $T^*M \setminus 0$ by the action of $\R_+$ by scalar multiplication, where $\Phi^t_g$ denotes the time-$t$ unit-speed geodesic flow on $S^*M$. $U \cap  \mathcal G_r^k$ is the complement of the image of $E_r$ via the projection onto $ \mathcal G^k_r$. Since the injectivity radii of metrics in $ \mathcal G_r^k$ are greater than $r$, there are no closed geodesics of period $r$ (actually $2r$) or less. Hence
    \[
        E_r \subset  \mathcal G^k_r \times [r,T] \times S^*M.
    \]
    
    Provided the metrics have sufficiently many continuous derivatives, and since the rank function is lower semi-continuous, $E_r$ is a closed subset of $ \mathcal G_r^k \times [r,T] \times S^*M$.
    Let $\pi$ denote the projection $ \mathcal G_r^k \times [r,T] \times S^*M \to  \mathcal G_r^k$. Since $[r,T] \times S^*M$ is compact, $\pi$ is a closed map and hence $\pi E_r$ is closed in $ \mathcal G_r^k$, hence $U \cap  \mathcal G_r^k$ is open in $ \mathcal G_r^k$.
\end{proof}

This section's proposition follows from Lemma \ref{lem: open-dense bumpy metric} and the following observation of White \cite{white2017bumpy}. We include a short proof for the sake of completeness.

\begin{lemma}
    Let $U \subset  \mathcal G^k$ be such that $U \cap  \mathcal G^j$ is open and dense in $ \mathcal G^j$ for all $j \geq k$. Then, $U \cap \mathcal G$ is open and dense in $\mathcal G$.
\end{lemma}

\begin{proof}
    We will use two facts:
    \begin{enumerate}
        \item The collection of all of the open sets of $ \mathcal G^j$ for $j \geq k$ form a basis for the $C^\infty$ topology on $\mathcal G$.
        \item $\mathcal G$ is dense in $ \mathcal G^j$ for each $j \geq k$.
    \end{enumerate}
    Clearly $U$ is open in $\mathcal G$, so it suffices to show it is dense. Suppose $V$ is any nonempty open neighborhood in $\mathcal G$, say containing an element $g$. By (1) there exists $j \geq k$ and an open neighborhood $V'$ of $g$ in $ \mathcal G^j$ with $V' \cap \mathcal G \subset V$. Now, $V' \cap U$ is open and nonempty in $ \mathcal G^j$ since $U \cap  \mathcal G^j$ is open and dense in $C^j$. By (2), there exists $g' \in \mathcal G$ with $g' \in V' \cap U$. In particular, $g' \in V \cap U$, as desired.
\end{proof}

\section{Geodesics and two useful perturbations of the metric}\label{sec: perturbation}
In this section, we introduce two distinct types of metric perturbations that will be employed in our proof. First, we construct a perturbation of the metric within the same isometry class as the original metric, which allows us to locally perturb the embedding of the submanifold.
\subsection{Perturbing the embedding of the submanifold} \label{sec: perturbation 1}
Fix a metric $g \in \mathcal G$. Let $\mathcal I(g)$ denote the set of all smooth Riemannian metrics in the same isometry class as $g$. As a consequence of the Arzel\`a-Ascoli theorem, $\mathcal I(g)$ is a closed subset of $\mathcal G$, and hence is completely metrizable in its own right. We can then consider residual subsets of $\mathcal I(g)$.

We will consider a class of perturbations generated by diffeomorphisms on $M$. The resulting metric perturbation will result in a metric that is in the same isometry class as the original. This might seem fruitless initially, but it allows us to effectively perturb the embedding of $\Sigma$ into $M$ while preserving the geometry of $M$. 

Before constructing our perturbations, consider a diffeomorphism $\kappa : M \to M$. $\kappa$ induces a canonical transformation $\tau : T^*M \to T^*M$ which can be written in canonical local coordinates as
\[
    \tau(x,\xi) = (\kappa(x), d\kappa(x)^{-t} \xi),
\]
where $d\kappa(x)^{-t}$ denote the inverse-transpose of the Jacobian matrix $d\kappa(x)$. 

\begin{proposition}
    Let $g$ be a Riemannian metric on $M$, let $\kappa : M \to M$ be a diffeomorphism and let $\tau : T^*M \to T^*M$ be the associated canonical transformation. If $g' = \kappa^* g$, then
    \[
        p_{g'} = p_g \circ \tau
    \]
    where $p_g$ and $p_{g'}$ are as in \eqref{eq: p_g}. It follows that $d\tau H_{p_{g'}} = H_{p_g}$, and that $\tau$ takes orbits of the geodesic flow on $(M,g')$ to orbits of the geodesic flow on $(M,g)$, and vice-versa.
\end{proposition}

\begin{proof}
    We note
    \begin{align*}
        p_g(\tau(x,\xi)) &= p_g(\kappa(x), d\kappa(x)^{-t} \xi) \\
        &= \sum_{i,j} g^{ij}(\kappa(x)) (d\kappa(x)^{-t} \xi)_i (d\kappa(x)^{-t} \xi)_j \\
        &= \sum_{i,j,k,\ell} g^{ij}(\kappa(x)) [d\kappa(x)^{-t}]_{i,k} [d\kappa(x)^{-t}]_{j,\ell} \xi_k \xi_\ell \\
        &= \sum_{k,\ell} \left[ d\kappa(x)^t g(\kappa(x)) d\kappa(x) \right]^{-1}_{k,\ell} \xi_k \xi_\ell \\
        &= \sum_{k,\ell} {g'}^{k\ell}(x) \xi_k \xi_\ell \\
        &= p_{g'}(x,\xi).
    \end{align*}
    Since $d\tau$ is a linear symplectomorphism $T_{(x,\xi)} T^*M \to T_{\tau(x,\xi)} T^*M$, we have, for each vector $v \in T_{(x,\xi)} T^*M$, and where $\sigma$ denotes the symplectic form on $T^*M$,
    \[
        \sigma(H_{p_{g'}}, v) = dp_{g'}(v) = dp_g (d\tau(v)) = \sigma(H_{p_g}, d\tau(v)) = \sigma(d\tau^{-1} H_{p_g}, v),
    \]
    from which the second part of the proposition follows. The third part follows quickly.
\end{proof}

We construct our perturbations as follows. Fix a smooth vector field $X$ on $M$ and consider the pullback
\[
    g_X = \exp(X)^*g
\]
of $g$ through the time-$1$ flow $\exp(X)$. Note, $\exp(X) : (M,g_X) \to (M,g)$ is an isometry by construction, and hence $g_X \in \mathcal I(g)$. Hence, to study $\Sigma$ in $(M,g_X)$, it suffices to study $\exp(X)(\Sigma)$ in the original space $(M,g)$. Since we will be studying the conormal bundle, we will also need to consider the induced canonical transformation
\[
    \tau_X(x,\xi) = (\exp(X)(x), (d\exp(X)(x))^{-t}\xi).
\]
We illustrate this situation with the commutative diagrams
\begin{equation}\label{eq: commutative diagrams}
    \begin{tikzcd}
        & \Sigma \arrow[dl] \arrow[rd] \\
        (M,g_X) \arrow[rr,"\exp(X)"] & & (M,g)
    \end{tikzcd}
    \qquad
    \begin{tikzcd}
        T^*M \arrow[r,"\tau_X"] \arrow[d] & T^*M \arrow[d] \\
        M \arrow[r,"\exp(X)"] & M
    \end{tikzcd}
\end{equation}
where the arrow $\Sigma \to (M,g)$ is not the embedding we started with.
By the proposition above, $\tau_X$ preserves the orbits of the geodesic flow between $(M,g_X)$ and $(M,g)$.

\begin{proposition} \label{prop: perturbation}
    Fix a metric $g \in \mathcal G$. About each point in $M$, there exists an open neighborhood $U \subset M$ and a smooth parametrization of vector fields
    \[
        F : \R^{n + n^2} \to \mathcal X(M)
    \]
    on $M$ for which $F(0) = 0$ and, for each $(x,\xi) \in \dot T^*U$
    \[
        \tau_{F}(x,\xi) : \R^{n + n^2} \to T^*M
    \]
    has full-rank differential for $(a,b)$ on a neighborhood of the origin.
\end{proposition}

\begin{proof}
    Fix $(x_0, \xi_0) \in \dot T^*M$ and consider local coordinates about $x_0$. Let $\chi$ be a smooth cutoff function on $M$ which vanishes outside the local coordinate patch and is identically $1$ on an open neighborhood $U$ of $x_0$. Given $a \in \R^n$ and $b \in \R^{n^2}$, we define
    \[
        F(a,b) = \chi(x) \sum_{j = 1}^n \left(a_j + \sum_{i = 1}^n x_i b_{ij} \right) \frac{\partial}{\partial x_j}.
    \]
   
For $x$ in a neighborhood of $x_0$,
$$
F(a,b)(x) 
=
\sum_{j=1}^n
\bigl(
  a_j 
  + 
  \sum_{i=1}^n b_{ij}\,x_i
\bigr)
\,\frac{\partial}{\partial x_j},
$$
where $a=(a_1,\dots,a_n)\in\mathbb{R}^n$ and $b=(b_{ij})\in\mathbb{R}^{n^2}.$ 
The ODE for $x(t)\in\mathbb{R}^n$ is
$$
\frac{d}{dt}x_j(t)
=
a_j +\sum_{i=1}^n b_{ij}\,x_i(t),
\quad
x(0)=x.
$$

Let $X(t)=(x_1(t),\dots,x_n(t))^T$, $A=(a_1,\dots,a_n)^T$, and let $B$ be the $n\times n$ matrix with $(B)_{j,i}=b_{ij}$. Then
\[
\frac{d}{dt}X(t) = A + B\,X(t).
\]
Its general solution is
\[
X(t)
=
e^{tB}\,x
+
\biggl(\int_0^t e^{(t-s)B}\,ds\biggr)\,A.
\]
At time $1$,
\[
X(1)
=
e^{B}\,x
+
\bigl(\int_0^1 e^{(1-s)B}\,ds\bigr)\,A
=:
e^{B}\,x + T(B)\,A.
\]

\medskip

Hence the basepoint of the flow is
\[
y = \exp(F(a,b))(x)
=
e^{B}\,x + T(B)\,A.
\]
The derivative w.r.t.\ $x$ is
\[
d\,\exp(F(a,b))(x) = e^B,
\quad
\Longrightarrow
\quad
(d\,\exp(F(a,b))(x))^{-t} = e^{-\,B^t}.
\]
Therefore,
\[
\eta 
=
\bigl(d\,\exp(F(a,b))(x)\bigr)^{-t}\,\xi
=
e^{-\,B^t}\,\xi,
\]
and the map on phase space is
\[
\tau_{F(a,b)}(x,\xi)
=
\bigl(e^B x + T(B)\,A,e^{-B^t}\,\xi\bigr).
\]

We evaluate the derivative of the map $ \tau_{F(a,b)}(x,\xi) $ at the point $ ((a,b),(x,\xi)) = ((0,0),(x,\xi)) $. The derivative is expressed in the usual block matrix form as follows:

\[
d_{(a,b)} \tau_{F(a,b)}(x,\xi) \bigg|_{\substack{(a,b) = (0,0) \\ (x,\xi) \text{ fixed}}} 
= 
\begin{pmatrix}
\left. \dfrac{\partial y}{\partial a} \right|_{(0,0)} & \left. \dfrac{\partial y}{\partial b} \right|_{(0,0)} \\
\left. \dfrac{\partial \eta}{\partial a} \right|_{(0,0)} & \left. \dfrac{\partial \eta}{\partial b} \right|_{(0,0)}
\end{pmatrix},
\]
where $ y = \exp(F(a,b))(x) $ and $ \eta = \left( d\,\exp(F(a,b))(x) \right)^{-t} \xi $.

    At $ (a,b) = (0,0) $, the map $ y $ depends linearly on $ a $ with a coefficient matrix equal to the identity matrix. Therefore,
    \[
    \left. \dfrac{\partial y}{\partial a} \right|_{(0,0)} = I_n.
    \]
    The component $ \eta $ does not depend on $ a $ at first order, hence 
    \[
    \left. \dfrac{\partial \eta}{\partial a} \right|_{(0,0)} = 0_{n \times n}.
    \]

It then suffices to show that the lower-right block of $d_{(a,b)}\tau_{F(a,b)}$ has full-rank.
At $(a,b)=(0,0)$, we have $B^t=0$, so
\[
\left.\frac{\partial \eta}{\partial b_{ij}}\right|_{(0,0)}
= -\,E_{i,j}\,\xi
= -\,\xi_j\,\mathbf{e}_i.
\]
In matrix form, the $(i,(k,\ell))$-entry of the $n\times n^2$ block is
\[
\begin{cases}
-\xi_\ell, & \text{if } i=k,\\
0,         & \text{otherwise}.
\end{cases}
\]
Hence the lower-right block is
\[
\begin{pmatrix}
-\xi_1 & 0 & \cdots & 0
 & -\xi_2 & 0 & \cdots & 0
 & \cdots
 & -\xi_n & 0 & \cdots & 0 \\[6pt]
0 & -\xi_1 & \cdots & 0
 & 0 & -\xi_2 & \cdots & 0
 & \cdots
 & 0 & -\xi_n & \cdots & 0 \\[6pt]
\vdots & & \ddots & 
 & & & & & 
 & & & & \\[6pt]
0 & 0 & \cdots & -\xi_1
 & 0 & 0 & \cdots & -\xi_2
 & \cdots
 & 0 & 0 & \cdots & -\xi_n
\end{pmatrix},
\]
which clearly has full rank as long as $\xi\neq 0.$ 
\end{proof}
{

\subsection{A Conformal Perturbation} 

{
We now construct a conformal perturbation of the metric that will allow us to make any small prescribed perturbation of the terminal data of integral curves of $\frac 1 2 H_{p}$ in $T^*M$, provided these integral curves do not close.

We first work locally. Consider a bounded, open, convex neighborhood of the unit line segment $[0,1] \times \{0\}$ in $\R^n$, and on it a Riemannian metric $g$ for which
\begin{equation}\label{eq: fermi conditions}
	g_{ij}(x) = \begin{bmatrix}
		1 + O(|x'|^2) & 0 \\
		0 & I + O(|x'|^2)
	\end{bmatrix}
	\qquad \text{ where } x = (x_1, x').
\end{equation}
One should think of the metric as having been obtained from Fermi coordinates about a geodesic segment $[0,1] \times \{0\}$. We consider a scaling
\[
	g(\epsilon;x) = g(\epsilon(x - e_1) + e_1)
\]
of the metric, where $e_1 = (1,0,\ldots,0) \in \R^n$, where $0 \leq \epsilon \leq 1$. Note, the scaled metric also satisfies the condition \eqref{eq: fermi conditions}, except that a factor of $\epsilon^2$ precedes the errors. From now on, we will implicitly include this scaling into the metric, even if its inclusion is not indicated by the notation.

Let $(x,\xi)$ be the corresponding canonical local coordinates for the cotangent bundle over our chart, and 
consider the symbol
\[
	p_g(x,\xi) = g^{ij}(x) \xi_i \xi_j
\]
and the associated Hamiltonian vector field
\[
	\frac 1 2 H_{p_g} = \frac 1 2 \left( \frac{\partial p}{\partial \xi_i} \frac{\partial}{\partial x^i} - \frac{\partial p}{\partial x^i} \frac{\partial}{\partial \xi_i} \right).
\]
Note, the curve $(x(t), \xi(t)) = (te_1, e_1)$ is the unique integral curve starting at $(0, e_1)$. 

We are interested in studying the effect of conformal perturbations of the metric---and hence of the Hamiltonian vector field---on the integral curve with the same initial data. We start by taking our choice of open set $U \subset \R^n$ for which
\[
	U \cap ([0,1] \times \{0\}) = (0,1) \times \{0\}.
\]
Then, we take a smooth function $f$ supported in $U$ and consider the associated conformal perturbation
\[
	g(\epsilon,s;x) = \frac{1}{1 + sf(x)} g(\epsilon;x) \qquad \text{and} \qquad p_{g(\epsilon,s)}(x,\xi) = (1 + sf(x)) p_g(x,\xi).
\]
For each $s$, let $(x(s;t), \xi(s;t))$ be the integral curve for $\frac 1 2 H_{(1 + sf)p}$ with $(x(s,0), \xi(s,0)) = (0,e_1)$. 
We wish to describe
\[
	\left. \frac \partial {\partial s_j} \right|_{s = 0} (x(s,t), \xi(s,t)).
\]
To do so, we write down Hamilton's equations
\begin{align*}
	\dot x^i &= \frac 1 2 (1 + sf) \frac{\partial p}{\partial \xi_i} \\
	\dot \xi_i &= - \frac 1 2 (1 + sf) \frac{\partial p}{\partial x^i} - \frac 1 2 p s \frac{\partial f}{\partial x^i},
\end{align*}
from which we obtain, at $s = 0$,
\begin{equation*}
\begin{split}
	\frac{\partial}{\partial t} \frac{\partial x^i}{\partial s} &= \frac 1 2 f \frac {\partial p} {\partial \xi_i} + \frac 1 2 \left( \frac {\partial^2 p} {\partial x^j \partial \xi_i} \frac{\partial x^j}{\partial s} + \frac {\partial^2 p} {\partial \xi_j \partial \xi_i} \frac{\partial \xi_j}{\partial s} \right) 
	\\
	\frac{\partial}{\partial t} \frac{\partial \xi_i}{\partial s} &= - \frac 1 2 f \frac{\partial p}{\partial x^i} - \frac 1 2 \left( \frac{\partial^2 p}{\partial x^j \partial x^i} \frac{\partial x^j}{\partial s} + \frac{\partial^2 p}{\partial \xi_j \partial x^i} \frac{\partial \xi_j}{\partial s} \right) - \frac 1 2 p \frac{\partial f}{\partial x^i}
\end{split}
\end{equation*}
which simplifies, by repeated application of \eqref{eq: fermi conditions} and by noting $p = 1$, to
\begin{equation} \label{eq: perturbed hamilton equations}
	\begin{split}
		\frac{\partial}{\partial t} \frac{\partial x^i}{\partial s} &= f \xi_i + \frac{\partial \xi_i}{\partial s} \\
		\frac{\partial}{\partial t} \frac{\partial \xi_i}{\partial s} &= - \frac 1 2 \frac{\partial^2 p}{\partial x^j \partial x^i} \frac{\partial x^j}{\partial s} - \frac 1 2 \frac{\partial f}{\partial x^i}
	\end{split}
\end{equation}
with initial data
\begin{equation}\label{eq: perturbed initial data}
	\frac{\partial x}{\partial s} = \frac{\partial \xi}{\partial s} = 0 \qquad \text{ at $t = 0$.}
\end{equation}

We consider two kinds of functions $f$. The first adjusts the distance between the ends of the geodesic.

\begin{lemma} Taking everything as above, let $f$ be a smooth function supported on the coordinate patch but not at $0$ nor $t_0 e_1$, and
\[
	\frac{\partial f}{\partial x^i}(te_1) = 0 \qquad \text{ for each $t \in [0,t_0]$ and $i \geq 2$.}
\]
Then,
\begin{align*}
	\frac{\partial x}{\partial s} &= \frac 1 2 \left( \int_0^{t} f(ue_1) \, du \right) e_1 \\
	\frac{\partial \xi}{\partial s} &= -\frac 1 2 \frac{\partial f}{\partial x^1}(te_1) e_1.
\end{align*}
\end{lemma}

\begin{proof}
One checks directly that the solution in the statement solves \eqref{eq: perturbed hamilton equations} with the initial data
\[
	\frac{\partial x}{\partial s} = \frac{\partial \xi}{\partial s} = 0 \qquad \text{ at $t = 0$.}
\]
The lemma follows by uniqueness.
\end{proof}

The second kind of perturbation adjusts the trajectory of the geodesic, and is the more important of the two. For this, we use functions $f$ of the kind where $f(te_1) = 0$ for each $t$. The solution to \eqref{eq: perturbed hamilton equations} and \eqref{eq: perturbed initial data}.

We wish to, if not solve, then estimate the solution to \eqref{eq: perturbed hamilton equations} with initial data $\frac{\partial x}{\partial s} = 0$ and $\frac{\partial \xi}{\partial s} = 0$ at $t = 0$ for arbitrarily small $\epsilon > 0$. In the $\epsilon = 0$ limit, the pesky $\frac{\partial^2 p}{\partial x^i \partial x^j}$ term vanishes and we have an explicit solution
\begin{equation}\label{eq: idealized solution}
	\begin{split}
		\frac{\partial \tilde x}{\partial s} &= - \frac 1 2 \int_0^t (t - u) \nabla f(ue_1) \, du  \\
		\frac{\partial \tilde \xi}{\partial s} &= - \frac 1 2 \int_0^t \nabla f(ue_1) \, du.
	\end{split}
\end{equation}
The next lemma shows the true solution is not too far from this one if $\epsilon > 0$ is small enough.

\begin{lemma}
	Let $f$ be a smooth function supported in $U$ for which $f(te_1) = 0$ for each $t$. Then, with $\tilde x$ and $\tilde \xi$ as in \eqref{eq: idealized solution}, we have
	\[
		\left|\frac{\partial x}{\partial s} - \frac{\partial \tilde x}{\partial s} \right| = O(\epsilon^2) \qquad \text{ and } \qquad \left|\frac{\partial \xi}{\partial s} - \frac{\partial \tilde \xi}{\partial s} \right| = O(\epsilon^2)
	\]
	uniformly for $t \in [0,1]$.
\end{lemma}

\begin{proof}
Set
\[
	u = \frac{\partial x}{\partial s} - \frac{\partial \tilde x}{\partial s} \qquad \text{ and } \qquad v = \frac{\partial \xi}{\partial s} - \frac{\partial \tilde \xi}{\partial s},
\]
let $\nabla^2_x p$ denote the Hessian matrix of $p$ in the $x$-variables only,
and note $u$ and $v$ solve the nearly Hamiltonian system
\[
	\begin{split}
		u' &= v \\
		v' &= - \frac 1 2 \nabla_x^2 p \frac{\partial x}{\partial s} - \frac 1 2 u,
	\end{split}
\]
with $u = v = 0$ at $t = 0$. Now, consider the energy
\[
	E(t) = \frac 1 2 |u(t)|^2 + |v(t)|^2
\]
and note
\begin{align*}
	E' &= \langle u, u' \rangle + 2 \langle v, v' \rangle = - \left\langle v, \nabla_x^2 p \frac{\partial x}{\partial s} \right\rangle,
\end{align*}
and hence
\[
	|E'| \leq \sup_{t \in [0,1]} \|\nabla_x^2 p\| \left| \frac{\partial x}{\partial s} \right| |v| \leq a b \epsilon^2 \sqrt{E} \leq \frac{1}{2} \left( ab \epsilon^4 + ab E \right)
\]
where $a$ and $b$ are constants for which
\[
	\|\nabla_x^2 p\| \leq a \epsilon^2 \qquad \text{ and } \qquad \left|\frac{\partial x}{\partial s} \right| \leq b
\]
for each $t \in [0,1]$ and $0 < \epsilon < 1$. By Gr\"onwall's inequality, and since $E(0) = 0$,
\[
	E(t) \leq \epsilon^4 
	e^{\frac{ab}{2} t} \qquad \text{ for $t \in [0,1]$}.
\]
The lemma follows.
\end{proof}

Next, we consider a $(2n - 1)$-parameter family of perturbations. Let $h_1, h_2, h_3$ all be smooth functions supported in $U$ satisfying the following conditions. For $h_1$, we assert
\[
	\frac{\partial h_1}{\partial x^i}(te_1) = 0 \qquad \text{for each $t \in [0,1]$ and $i \geq 2$}
\]
and
\[
	\frac 1 2 \int_0^1 h_1(te_1) \, dt = 1.
\]
For $h_2$ and $h_3$, we assert
\begin{equation}\label{eq: direction-changing perturbation}
	-\frac 1 2 \int_0^1 h_2(te_1) \, dt = 1 \qquad \text{ and } \qquad -\frac 1 2 \int_0^1 (1 - t) h_2(te_1) \, dt = 0
\end{equation}
and
\[
	-\frac 1 2 \int_0^1 h_3(te_1) \, dt = 0 \qquad \text{ and } \qquad -\frac 1 2 \int_0^1 (1 - t) h_3(te_1) \, dt = 1.
\]
Then, consider the function
\begin{equation}\label{eq: final conformal perturbation}
	f_s(x) = s_1 h_1(x) + \sum_{i = 2}^n s_i x_i h_2(x) + \sum_{i = 2}^n s_{i+n-1} x_{i} h_3(x)
\end{equation}
for $s = (s_1, \ldots, s_{2n-1}) \in \R^{2n-1}$ along with the corresponding perturbation
\[
	(1 + f_s)p
\]
of the symbol. Combined with the previous lemmas, we obtain:

\begin{lemma} With $f_s$ as above, set $p_s(x,\xi) = (1 + f_s(x))p(x,\xi)$, and let $(x(\epsilon,s;t), \xi(\epsilon,s;t))$ be the family of integral curves to the Hamiltonian vector field $\frac 1 2 H_{p_s}$ with initial conditions
\[
	(x(\epsilon,s;0), \xi(\epsilon,s;0)) = (0,e_1).
\]
Then,
\[
	D_s (x,\xi) = \begin{bmatrix}
		I_n & 0 \\
		0 & 0 \\
		0 & I_{n-1}
	\end{bmatrix}
	+ O(\epsilon^2) \qquad \text{ at $s = 0$, $t = 1$} 
\]
\end{lemma}

We note that the geometer's geodesic flow for the metric $g$, i.e. the flow along the Hamiltonian vector field $\tfrac 1 2 H_{p}$, is homogeneous in the sense that
\begin{equation}\label{eq: flow homogeneity}
	\exp(\tfrac 1 2 H_{p})(0, \tau e_1) = (\tau e_1, \tau e_1),
\end{equation}
and hence
\[
	\left.\frac{\partial}{\partial \tau}\right|_{\tau = 1} \exp(\tfrac 1 2 H_{p})(0, \tau e_1) = (e_1, e_1).
\]
Combining this observation with the previous lemma yields:

\begin{lemma} \label{lem: local full rank}
	With $f_s$ as above, we have
	\[
		D_{\xi_1,s} \exp\big(\tfrac 1 2 H_{(1 + f_s)p}\big)(0,\xi_1 e_1) = \begin{bmatrix}
			* & I_{n} & 0 \\
			1 & 0 & 0 \\
			0 & 0 & I_{n-1}
		\end{bmatrix} + O(\epsilon^2).
	\]
	at $(\xi_1,s) = (1,0)$. That is, $\exp(\frac 1 2 H_{(1 + f_s)p})(0,\xi_1 e_1)$ has surjective differential at $(\xi_1,s) = (1,0)$ for all sufficiently small $\epsilon > 0$.
\end{lemma}

We are ready to transfer this lemma to our manifold.

\begin{proposition}\label{prop: conformal perturbation}
Suppose $(x_0, \xi_0) \in \dot T^*M$ satisfies
\begin{equation}\label{eq: non-closed condition}
	\exp(\tfrac t 2 H_p)(x_0, \xi_0) \neq (x_0, \xi_0) \qquad \text{ for all } t \in (0,1].
\end{equation}
Then, there exists an open neighborhood $U$ in $M$, an open neighborhood $S$ of the origin in $\R^{2n-1}$, and a smooth function $(s,x) \mapsto f_s(x)$ on $S \times U$ whose $x$-support is contained in $U$, such that
\[
	\exp(\tfrac 1 2 H_{(1 + f_s)p})(x_0, \tau\xi_0)
\]
has surjective differential $\R^{1 + (2n-1)} \to T_{(y_0,\eta_0)}T^*M$ at $(\tau,s) = (1,0)$, where
\[
	(y_0,\eta_0) = \exp(\tfrac 1 2 H_p)(x_0, \xi_0).
\]
\end{proposition}

\begin{proof}
	Let $\gamma(t)$ be the geodesic in $M$ traced by the spacial component of the integral curve $\exp(\frac t 2 H_p)(x_0, \xi_0)$. The condition \eqref{eq: non-closed condition} ensures $\gamma$ intersects itself only finitely many times. The argument is perhaps a routine Riemannian geometry exercise, so we just sketch the steps.
First, since $\gamma$ is locally an embedding about each point and $[0,1]$ is compact, pairs of times $t \neq t'$ in $[0,1]$ at which $\gamma(t) = \gamma(t')$ are separated by a positive distance. If there are infinitely many such pairs, one takes a convergent sequence $(t_k, t_k') \to (t,t')$ and notes $t \neq t'$. Since $\gamma(t_k) = \gamma(t_k')$ for each $k$, we deduce $\gamma'(t) = \pm \gamma'(t')$. If the sign is positive, it contradicts \eqref{eq: non-closed condition}. If the sign is negative, the geodesic is ill-defined at $\frac{t + t'}{2}$.

	Since $\gamma$ has finitely many self-intersections, there exists $\epsilon > 0$ for which $\gamma : (1 - \epsilon, 1) \to M$ does not intersect $\gamma([0,1-\epsilon] \cup \{1\})$. Select Fermi normal coordinates about the geodesic segment $\gamma([1-\epsilon,1])$ satisfying \eqref{eq: fermi conditions}. Consider an open neighborhood $U$ intersecting $\gamma$ only along the segment $\gamma((1 - \epsilon, 1))$. Any perturbation of the metric supported in $U$ will leave $(x(1 - \epsilon), \xi(1 - \epsilon))$ unchanged. The lemma follows after taking $\epsilon > 0$ small, rescaling the interval $[1-\epsilon,1]$ to $[0,1]$, using \eqref{eq: flow homogeneity}, and Lemma \ref{lem: local full rank}.
\end{proof}
}

\begin{figure}
    \centering
    \includegraphics[width=0.45\textwidth]{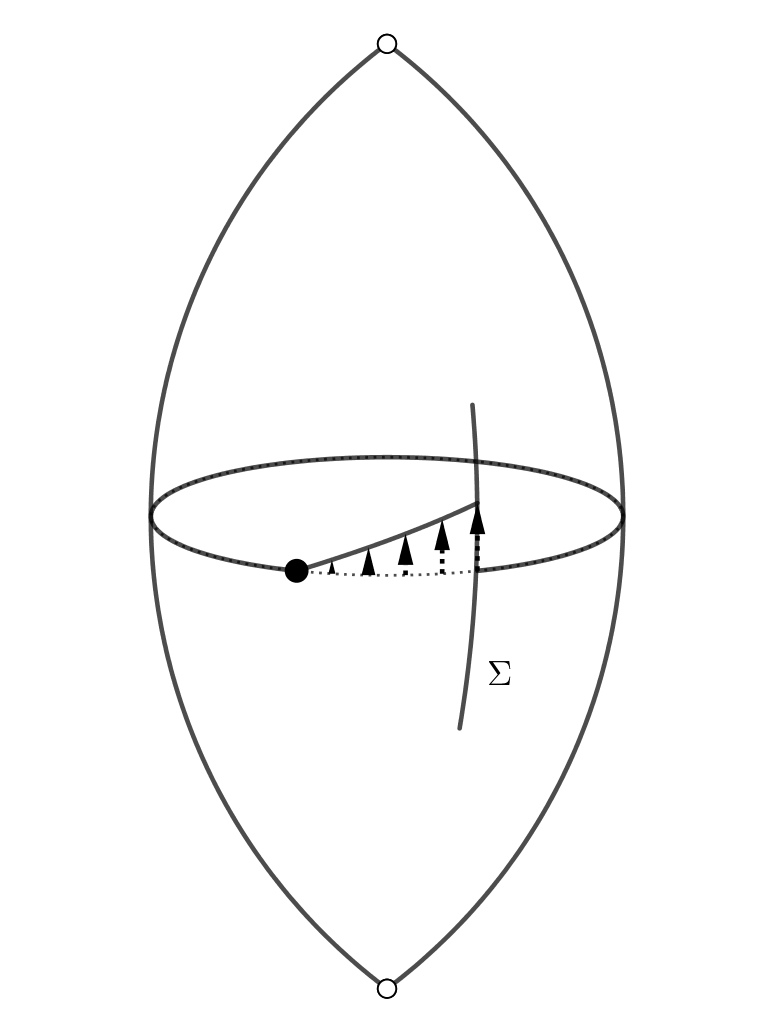}
    \includegraphics[width=0.45\textwidth]{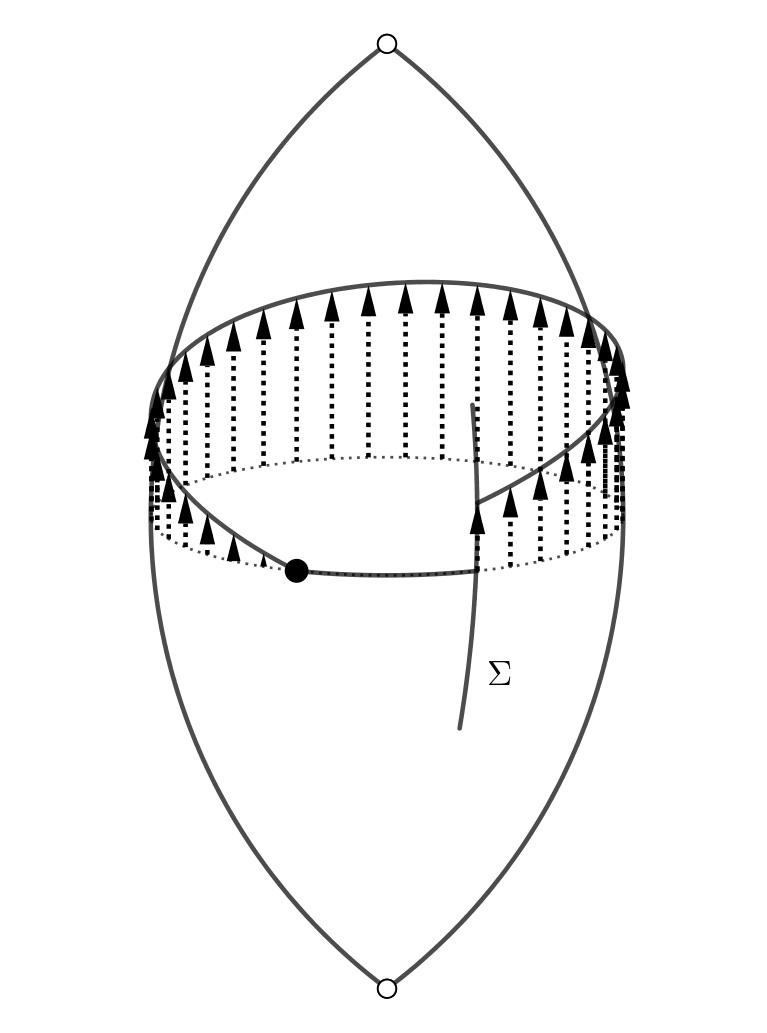}
    \caption{On the left, we see the effect of a sharp, localized perturbation of the metric on the first period of the equatorial geodesic. On the right, we see the effect of the perturbation on the second period.}
    \label{fig: footballs}
\end{figure}

\begin{remark}\label{rmk: footballs}
    The condition \eqref{eq: non-closed condition} in Proposition \ref{prop: conformal perturbation} is necessary. A closed geodesic necessarily passes through the support of a metric perturbation multiple times, and the effect of the perturbation can cancel itself on the second pass-through.
    To see how this is possible, consider a surface $M$ obtained by taking the quotient of the sphere, minus the poles, by the action of a half-turn about the vertical axis (see Figure \ref{fig: footballs}). Let $\Sigma$ be a segment of a longitudinal curve that intersects the equator. We consider a conformal perturbation of the metric as constructed from \eqref{eq: direction-changing perturbation}, where for the sake of this thought experiment, we will scale all the way to the limit $\epsilon \to 0$. The variation of the geodesic is a solution to a non-homogeneous Jacobi equation which is identically zero upon the second return to $\Sigma$.
\end{remark}

}

\section{Generic submanifolds have no closed normal geodesics}\label{sec: no closed normal}

{
The main proposition of this section shows that, for most metrics, an embedded submanifold intersects no closed geodesics conormally.

\begin{proposition}\label{prop: generically no conormal closed geodesics}
    For each $T > 0$, there exists an open-dense set of metrics $g$ such that, for each $(x,\xi) \in \dot N^*\Sigma$ with $p_g(x,\xi) \leq T$ and for each $t \in (0,1]$,
    \[
        \exp(\tfrac t 2 H_{p_g})(x,\xi) \neq (x,\xi).
    \]
\end{proposition}

Taking the countable intersection over $T \in \N$ yields a residual set of metrics for which the proposition holds without constraints on $p_g(x,\xi)$, but this is not what we need here.

Our starting point is Proposition \ref{prop: bumpy metric}, which gives an open-dense set, say $\mathcal G_T$, of metrics for which there are only finitely many $(x,\xi) \in \dot T^*M$ with $p_g(x,\xi) \leq T$ and
\[
    \exp(\tfrac t 2 H_{p_g})(x,\xi) = (x,\xi).
\]
We first argue:

\begin{lemma}
    The exceptional set of metrics $g$ for which there exists $(x,\xi) \in \dot N^* \Sigma$ with $p_g(x,\xi) \leq T$ and
    \[
        \exp(\tfrac t 2 H_{p_g})(x,\xi) = (x,\xi) \qquad \text{ for some } t \in (0,1]
    \]
    is closed in $\mathcal G$, and hence also in $\mathcal G_T$.
\end{lemma}

\begin{proof}
    Consider the exceptional set $E \subset \mathcal G \times (0,1] \times \dot N^*\Sigma$ of elements $(g,t,x,\xi)$ for which
    \[
        p_g(x,\xi) \leq T, \qquad \text{ and } \qquad \exp(\tfrac t 2 H_{p_g})(x,\xi) = (x,\xi).
    \]
    Let $\pi : E \to G$ be the standard projection. We claim $\pi(E)$ is closed. Consider a sequence of metrics $g_n \in \pi(E)$ that converges to $g \in \mathcal G$. Let $(g_n, t_n, x_n, \xi_n) \in E$ be some element in $E$ that projects to $g_n$. Let $C$ be a constant for which
    \[
        p_{g_n} \leq C p_g \qquad \text{ for each } n \in \N,
    \]
    and note that since $p_{g_n}(x_n, \xi_n) \leq T$ for each $n$,
    \[
        p_g(x_n, \xi_n) \leq C T \qquad \text{ for each $n \in \N$}.
    \]
    Refine the sequence so that $(x_n, \xi_n) \to (x,\xi) \in N^*\Sigma$ and note
    \[
        p_g(x,\xi) = \lim_{n \to \infty} p_{g_n}(x_n, \xi_n) \leq T.
    \]
    Next, let $\epsilon > 0$ bound the injectivity radii of $g_n$ and $g$ from below, and note that if
    \[
        \exp(\tfrac {t_n} 2 H_{p_{g_n}})(x_n, \xi_n) = (x_n, \xi_n),
    \]
    then $t_n p_{g_n}(x_n, \xi_n) \geq \epsilon$. From this, we conclude
    \[
        t_n \in [\tfrac \epsilon T, 1] \qquad \text{ and } p_{g_n}(x_n, \xi_n) \geq \epsilon.
    \]
    We refine our sequence further so that $t_n$ converges to $t \in [\tfrac \epsilon T, 1]$, and note in particular that
    \[
        t \in (0,1] \qquad \text{ and } \qquad 0 < p_g(x,\xi) \leq T.
    \]
    Finally, $\exp(\frac t 2 H_{p_g})(x,\xi) = (x,\xi)$ too, and we conclude $g \in \pi(E)$, as desired.
\end{proof}

Secondly, we will use the perturbation introduced in Section \ref{sec: perturbation 1} to show:

\begin{lemma}
    The exceptional set of metrics $g$ for which there exists $(x,\xi) \in \dot N^* \Sigma$ with $p_g(x,\xi) \leq T$ and
    \[
        \exp(\tfrac t 2 H_{p_g})(x,\xi) = (x,\xi) \qquad \text{ for some } t \in (0,1]
    \]
    has dense complement in $\mathcal G_T$.
\end{lemma}

Once proved, the two lemmas complete the proof of Proposition \ref{prop: generically no conormal closed geodesics}.

\begin{proof}
    Fix $g \in \mathcal G_T$ and let
    \begin{multline*}
        \Gamma_T = \{(x,\xi) \in \dot T^*M : 0 < p_g(x,\xi) \leq T \\
        \text{ and } \exp(\tfrac t 2 H_{p_g})(x,\xi) = (x,\xi) \text{ for some $t \in (0,1]$}\}.
    \end{multline*}
    Note, $\Gamma_T$ is the finite disjoint union of $2$-dimensional connected components, and $\Gamma_T \cap \dot N^*\Sigma$ is a disjoint union of $1$-dimensional connected components, each of which lies tangent to the radial vector field in the fiber of $\dot T^*M$. Let $x_0 \in M$ be the projection of such a component to $M$. Consider the parametrized family of vector fields $F : \R^{n + n^2} \to \mathcal X(M)$ centered at $x_0$ as in Proposition \ref{prop: perturbation}. By the proposition and Thom's transversality theorem, there is a residual set of parameters $(a,b) \in \R^{n + n^2}$ for which $\tau_{F(a,b)}(\dot N^*\Sigma)$ intersects $\Gamma_T$ transversally. If the intersection is nonempty, it has dimension $1$, and yet
    \[
        2n = \dim T^*M = \dim(\dot N^*\Sigma) + \dim(\Gamma_T) - 1 = n + 1,
    \]
    which is an absurdity for $n \geq 2$. Hence, for most $(a,b)$, the intersection is empty. By the discussion preceding Proposition \ref{prop: perturbation}, the pullback metric
    \[
        \exp(F(a,b))^*g
    \]
    admits no intersection of $\dot N^*\Sigma$ and $\Gamma_T$ at $x_0$ for a dense set of $(a,b) \in \R^{n + n^2}$. Note, $\exp(F(a,b))^*g$ is in the same isometry class as $g$, and hence is also a member of $\mathcal G_T$, and can be made close to $g$ provide $(a,b)$ is close to the origin. We are done after repeating the procedure for each of the finitely many sites of intersection.
\end{proof}
}

\begin{figure}
\begin{tikzpicture}[scale=0.67]

    \draw[->, thick, dashed] (3,1.4) arc[start angle=405, end angle=45, x radius=3, y radius=2];
    \node[label=$\gamma$] at (3,1.4) {};

    \draw[domain=1:6,samples=100,thick] plot (\x,0);
    \node[label=$\Sigma$] at (5,0) {};
\end{tikzpicture}
\hfill
\begin{tikzpicture}[scale=0.67]

    \draw[->, thick, dashed] (3,1.4) arc[start angle=405, end angle=45, x radius=3, y radius=2];
    \node[label=$\gamma$] at (3,1.4) {};

    \draw[domain=1:6,samples=100,thick] plot (\x,{0.5*sin(\x r + 100)});

    \node[label=$e^{F(a,b)}(\Sigma)$] at (5.5,0.5) {};
\end{tikzpicture}
\caption{A diffeomorphism $\exp(F(a,b)) : M \to M$, for many choices of $(a,b)$, prevents the perpendicular intersection of $\Sigma$ and a closed geodesic $\gamma$.}
\label{fig: closed geodesic}
\end{figure}
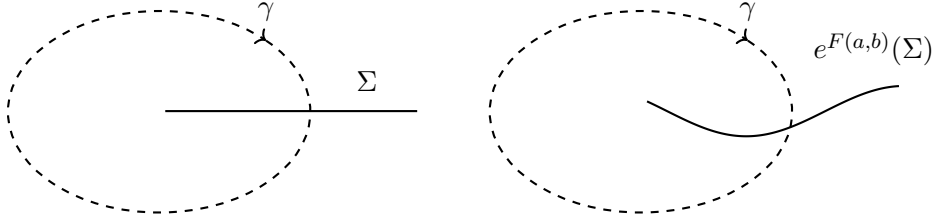

}

\section{Proof of Theorem \ref{thm: main}}\label{sec: proof of main}

{

Let $\mathcal G_T$ denote the set of metrics $g$ for which there are only finitely many geodesics that close by time $T$. Proposition \ref{prop: bumpy metric} ensures $\mathcal G_T$ is open-dense in the full space of metrics $\mathcal G$. Let $\mathcal G_T' \subset \mathcal G_T$ denote the subset of metrics for which there are no covectors $(x,\xi) \in \dot N^*\Sigma$ such that
\[
    p_g(x,\xi) \leq T \qquad \text{ and } \qquad \exp(\tfrac t 2 H_{p_g})(x,\xi) = (x,\xi) \text{ for some $t \in (0,1]$}.
\]
Proposition \ref{prop: generically no conormal closed geodesics} ensures $\mathcal G_T'$ is open-dense in $\mathcal G_T$, and hence also open-dense in $\mathcal G$. The next proposition concerns the set $\mathcal G_T'' \subset \mathcal G_T'$ of metrics for which
\[
    \exp(\tfrac 1 2 H_{p_g}) : \dot N^*\Sigma \to T^*M \text{ is transversal to } \dot N^*\Sigma \cap p_g^{-1}[0,\tfrac T 4],
\]
as required by Theorem \ref{thm: main}. 

\begin{proposition}\label{prop: almost there}
    The set of metrics $\mathcal G_T''$ is residual in $\mathcal G_T'$, and hence residual in $\mathcal G$.
\end{proposition}

Before we prove this last proposition, we first see how it implies Theorem \ref{thm: main}.

\begin{proof}[Proof of Theorem \ref{thm: main}]
    The countable intersection
    \[
        \bigcap_{N \in \N} \mathcal G_N''
    \]
    is also residual in $\mathcal G$. Furthermore, for each metric in the intersection,
    \[
        \exp(\tfrac 1 2 H_{p_g}) : \dot N^*\Sigma \to T^*M \text{ is transversal to } \dot N^*\Sigma,
    \]
    as desired.
\end{proof}
}

To prove Proposition \ref{prop: almost there}, we require a lemma that will take the perturbations we have established in Section \ref{sec: perturbation} and produce the desired residual set of metrics. We defer the proof until the end of this section.

\begin{lemma}\label{lem: machinery} Let $B$ be a Baire space, $X$ and $Y$ smooth manifolds, and $Z \subset Y$ an embedded submanifold. Let
\[
	\Phi : B \to C^\infty(X,Y)
\]
be a continuous map. Suppose for each $(b_0,x_0) \in B \times X$ for which $\Phi(b)$ is non-transversal to $Z$ at $x_0$, there exists a neighborhood $S$ of the origin in $\R^n$, an open neighborhood $U$ of $b_0$, and a continuous map $\beta : U \times S \to B$ such that:
\begin{enumerate}
	\item $\beta(b,0) = b$ for each $b \in U$.
	\item If $\Psi(b,s,x) = \Phi(\beta(b,s))x$, then $b \mapsto \Psi(b,\cdot, \cdot)$ is a continuous mapping $U \to C^\infty(S \times X, Y)$.
	\item $(s,x) \mapsto \Psi(b_0,s,x)$ is transversal to $Z$ at $(0,x_0)$.
\end{enumerate}
Then, $\{ b \in B : \Phi(b) \text{ is transversal to } Z \}$ is residual in $B$.
\end{lemma}

{
We are now ready to prove Proposition \ref{prop: almost there}.

\begin{proof}[Proof of Proposition \ref{prop: almost there}] 
    Fix $g_0 \in \mathcal G_T'$ and consider the open neighborhood of metrics
    \[
        \mathcal G_T'(g_0) := \{g \in \mathcal G_T' : \tfrac 1 2 p_{g_0} < p_g < 2p_{g_0} \}.
    \]
    Since $\mathcal G_T'(g_0)$ is an open subset of $\mathcal G_T'$, which is itself an open subset of the completely metrizable space $\mathcal G$, $\mathcal G_T'(g_0)$ is Baire, and this will be taken to be the space $B$ of the lemma. The roles of $X$, $Y$, and $Z$ of the lemma will be filled by $\dot N^*\Sigma$, $T^*M$, and $\dot N^*\Sigma \cap p_{g_0}^{-1}[0,T/2]$, respectively. The role of $\Phi : B \to C^\infty(X,Y)$ will be filled by the map
    \[
        g \mapsto \exp(\tfrac 1 2 H_{p_g}) \in C^\infty(\dot N^*\Sigma, \dot T^*M).
    \]
    Now, suppose $(g,x,\xi) \in \mathcal G_T'(g_0) \times \dot N^*\Sigma$ is a point at which this map is non-transversal to $\dot N^*\Sigma \cap p_{g_0}^{-1}[0,T/2]$. Then, we note that
    \[
        \exp(\tfrac t 2 H_{p_g}(x,\xi)) \neq (x,\xi) \qquad \text{ for each } t \in (0,1],
    \]
    and so we let $f_s$ with $s \in \R^{2n-1}$ be as in Proposition \ref{prop: conformal perturbation} and set
    \[
        \beta(g,s) = (1 + f_s)^{-1}g.
    \]
    Again since $\mathcal G_T'(g_0)$ is an open subset of $\mathcal G$ and since $f_0 = 0$, we may restrict $\beta$ to an open neighborhood $U$ of $g$ and an open neighborhood $S$ of $0$ in $\R^{2n-1}$ to realize $\beta$ as a continuous and well-defined mapping $U \times S \to \mathcal G_T'(g_0)$. Finally, by Proposition \ref{prop: conformal perturbation},
    \[
        (s,x,\xi) \mapsto \exp(\tfrac 1 2 H_{(1 + f_s)p_g})(x,\xi)
    \]
    is transversal to $\dot N^*\Sigma \cap p_{g_0}^{-1}[0,T/2]$. We conclude by Lemma \ref{lem: machinery} that the set of $g \in \mathcal G_T'(g_0)$ for which
    \[
        \exp(\tfrac 1 2 H_{p_g}) : \dot N^*\Sigma \to T^*M \text{ is transversal to } \dot N^*\Sigma \cap p_{g_0}^{-1}[0,T/2]
    \]
    is residual in $\mathcal G_T'(g_0)$. The proposition follows.
\end{proof}
}

We now prove Lemma \ref{lem: machinery}. First, we prove a preliminary lemma.

\begin{lemma}\label{lem: preliminary machinery}
    Suppose $B$ is a Baire space and $X$ a smooth manifold. Suppose $E$ is a closed subset of $B \times X$, and that for each $(b_0, x_0) \in E$, there exists an open neighborhood $U$ of $b_0$ and a compact neighborhood $K$ of $x_0$ for which
    \[
        \{ b \in U : (b,x) \in E \text{ for some } x \in K\}
    \]
    is nowhere dense in $U$. Then,
    \[
        \{ b \in B : (b,x) \in E \text{ for some } x \in X \}
    \]
    is meager in $B$.
\end{lemma}

\begin{proof}
    Our first task is to untether the dependence of $K$ from $U$. So, consider an arbitrary compact subset $K \subset X$. We note $E \cap (\{b_0\} \times K)$ is compact, and hence admits a finite cover
	\[
		U_1 \times K_1^\circ, \ \ldots, \ U_m \times K_m^\circ
	\]
	where $U_i$ and $K_i$ are picked by the lemma. Now set $U = U_1 \cap \cdots \cap U_m$, and note that $U$ remains an open neighborhood of $b_0$. At the same time,
	\begin{multline*}
		\{b \in U : (b,x) \in E \text{ for some } x \in K \} \\
		\subset \bigcup_{i = 1}^m \{b \in U : (b,x) \in E \text{ for some } x \in K_i \},
	\end{multline*}
	the right side being a finite union of nowhere dense sets, which is again nowhere dense. We have just proved: For each $b_0 \in B$ and each compact set $K \subset X$, there exists an open neighborhood $U$ of $b_0$ for which
	\[
		\{b \in U : (b,x) \in E \text{ for some } x \in K \}
	\]
    is nowhere dense in $U$. 
    
	Let $\pi : B \times K \to B$ be the projection, and note 
	\[
		\pi(E \cap (B \times K)) = \{b \in B : (b,x) \in E \text{ for some } x \in K \}
	\]
	is closed ($E$ is closed and $\pi$ is a closed mapping by the tube lemma) and also has dense complement by the argument above. It follows $\pi(E \cap (B \times K))$ is nowhere dense in $B$. Finally, we take a countable exhaustion of $X$ by compact sets $K_1, K_2, \ldots$, and note
	\[
		\pi(E) = \bigcup_j \pi(E \cap (B \times K_j))
	\]
	is the countable union of nowhere dense sets, and hence is meager.
\end{proof}

\begin{proof}[Proof of Lemma \ref{lem: machinery}]
	We claim that, for each $(b_0, x_0) \in B \times X$ for which $\Phi(b_0)$ is non-transversal to $Z$ at $x_0$, there exists an open neighborhood $U$ of $b_0$ and a compact neighborhood $K$ of $x_0$ in $X$ for which
	\[
		\{b \in U : \Phi(b) \text{ is non-transversal to $Z$ at some $x \in K$} \}
	\]
	is nowhere dense in $U$. We will be done by Lemma \ref{lem: preliminary machinery}.
    For bookkeeping purposes, let $E$ be the exceptional set of points $(b,x)$ for which $\Phi(b) : X \to Y$ is non-transversal to $Z$ at $x$.
    
	Fix $(b_0, x_0) \in E$ and let $S$, $U$, and $\beta$ be as in the statement. By (3), $(s,x) \mapsto \Psi(b_0, s, x)$ is transversal to $Z$ at $(0,x_0)$. There exist compact neighborhoods $S'$ of the origin in $S$ and $K$ of $x_0$ in $X$ so that $\Psi(b_0, \cdot, \cdot) : S' \times K \to Y$ is transversal to $Z$ on $S' \times K$. By (2) and the tube lemma,
	\[
		\{b \in U : \text{$\Psi(b, \cdot, \cdot)$ is non-transversal to $Z$ on $S' \times K$} \}
	\]
	is closed, and does not contain $b_0$. Hence, after perhaps shrinking $U$, we ensure that $\Psi(b, \cdot, \cdot)$ is transversal to $Z$ on $S' \times X$ for each $b \in U$. We will prove our claim for these choices of $U$ and $K$. First, we note
	\[
		\{b \in U : \Phi(b) \text{ is non-transversal to $Z$ at some $x \in K$}\} = \pi(E \cap (U \times K))
	\]
	and hence is closed. We need only show its complement is dense. To this end, fix $b$ in the set. Then, by the Thom transversality theorem, the set
	\[
		\{s \in S' : \Psi(b, s, \cdot) = \Phi(\beta(b,s)) \text{ is transversal to $Z$ on $K$} \}
	\]
	is dense in $S'$. In particular, by (1) and the continuity of $\beta$, every open neighborhood of $b$ contains $\beta(b,s)$ for some element $s$ from this set. At the same time,
	\[
		\beta(b,s) \not\in \pi(E \cap (U \times K)).
	\]
	This concludes the proof of the claim, and the theorem.
\end{proof}

\section{Further Problems and Open Questions}\label{sec: further problems}

In this paper, we have demonstrated that for a residual set of metrics, the improved Kuznecov remainders hold. We prove this by restricting our attention to an isometry class of a bumpy metric. Nevertheless, several intriguing questions remain open. We now describe some further problems that naturally arise from our work.

\begin{enumerate}
    \item \textbf{Generic metrics in a conformal class.}  
    In our analysis the genericity results were obtained essentially by starting with a bumpy metric. A natural question is whether similar genericity results hold when one restricts to a conformal class. In other words, given a conformal class $[g]$ on $M$, is there a residual subset of metrics in $[g]$ for which the improved spectral asymptotics and the absence of closed conormal geodesics hold? Since conformal deformations modify the metric by a scalar function, they affect the principal symbol in a controlled way; however, the impact on the associated geodesic flow may be more subtle. New techniques may be required to study these deformations and to play the role of the bumpy metric theorem to obtain a corresponding genericity result.

   \item \textbf{Generic embeddings for a fixed bumpy metric.}   In our proof we begin with a bumpy metric and perturb the embedding so that no periodic geodesic meets the submanifold orthogonally.  We subsequently apply a conformal perturbation of the metric to achieve full transversality and thus obtain improved Kuznecov remainders.  A natural question is whether that last step is needed: for a fixed bumpy metric on $M$, do generic embeddings $\sigma\colon\Sigma\to M$ already guarantee an improved remainder in the Kuznecov formula?  The bumpy condition itself is essential—on the standard $d$-sphere, \emph{no} embedding of any submanifold can improve the Kuznecov remainder.

  \item \textbf{More general classes of metrics and Hamiltonian systems.}  
A natural question is whether our genericity results (or even the basic Kuznecov formula) extend beyond Riemannian metrics. For example, one may study Finsler metrics, which allow non-quadratic norms, or more general Hamiltonian systems (e.g., with a potential term or magnetic field). Although Kuznecov-type formulas are well established for the Riemannian Laplacian, it remains open whether analogous formulas and eigenfunction estimates hold in these broader settings. Developing the necessary spectral theory in these contexts is a challenging open problem.

\end{enumerate}

We hope that addressing these questions will lead to a deeper understanding of the interplay between geometry, dynamics, and spectral theory.

\bibliography{references}{} 
\bibliographystyle{alpha}

\end{document}